\documentclass[
    a4paper,
    DIV=14,
    abstract=true,
    numbers=noenddot
]{scrartcl}

\usepackage[referable]{threeparttablex}
\usepackage{multirow}
\usepackage{tabularx}
\usepackage{float}
\usepackage{nicematrix}
\usepackage{booktabs}
\usepackage{multirow}
\usepackage{physics,lineno,amsmath,hyperref,amssymb,graphicx,xcolor,amsthm}
\usepackage[normalem]{ulem}
\usepackage{cancel}
\usepackage[T1]{fontenc}
\usepackage{authblk} % Package for author affiliations
\usepackage{orcidlink}

\newtheorem{theorem}{Theorem}[section]
\newtheorem{lemma}{Lemma}

\newtheorem{proposition}[theorem]{Proposition}
\newtheorem{remark}{Remark}

\providecommand{\keywords}[1]
{
  \small	
  \textbf{\textit{Keywords---}} #1
}

\title{Preservation of structural properties of the \texorpdfstring{CIR}{cir} model by  \texorpdfstring{$\theta$}{teta}-Milstein schemes}

\author[1]{S. Llamazares-Elias\,\orcidlink{0000-0001-9219-6749}}
\author[2]{A. Tocino\,\orcidlink{0000-0002-7910-1570}}
\affil[1]{Department of Mathematics, University of Salamanca, Spain\\\texttt{samirllamazares@usal.es}}
\affil[2]{Department of Mathematics, University of Salamanca, Spain\\\texttt{bacon@usal.es}}

\date{}
\begin{document}
\maketitle

\begin{abstract}
 The ability of $\theta$-Milstein methods with $\theta\ge 1$ to capture the non-negativity and the mean-reversion property of the exact solution of the CIR model is shown. 
 In addition, the order of convergence and the preservation of the long-term variance is studied. These theoretical results are illustrated with numerical examples.
\end{abstract}

\section{Introduction}\label{secintro}

Stochastic differential equations (SDEs)
\begin{equation}
\label{sde}
dX(t)=f(X(t))\,dt+g(X(t))\,dW(t),
\end{equation}
where $W_t$ is a standard Wiener process,
are an important tool in many fields due to their applications to modelling dynamical phenomena. In Finance, several models have been proposed to describe the changes of interest rates over time, see e.g. \cite{shreve} and the references therein.  In this work, we focus on the Cox-Ingersoll-Ross (CIR) model \cite{cox}, that describes the interest rate as the solution to 
\begin{equation}
\label{CIR}
    dX(t)=\alpha(\mu-X(t))\,dt+\sigma
\sqrt{X(t)}\,dW(t)
\end{equation}
where $\alpha,\mu$ and $\sigma$ are positive constants. Due to its nonlipschitzian diffusion coefficient, general existence theorems \cite{oksendal} do not apply to the CIR model. Instead, specific results showing the existence and uniqueness of a non-negative (a.e.) strong solution $X(t)$, called the CIR process, can be found in the literature, \cite{mao,yamada}. 
In addition, see \cite{feller}, if $X(0)>0$ and the parameters of the equation satisfy the Feller condition 
$2\alpha\theta \ge\sigma^2$  then $ X(t)>0$ for all $t\in[0,\infty) $, i.e., the solution remains positive when it starts positive.

Since the interest rate reverts to its long-term mean, mean reversion is a desirable property of interest rate models. The CIR model is mean reverting since the drift term represents a force pulling the interest rate towards its long-term mean \cite{shreve}:
%(when $X(t)>\mu$, the drift term of \eqref{CIR} is negative, which drives $X(t)$ back towards $\mu$ and when $X(t)<\mu$, the drift term is positive, which drives $X(t)$ forward towards $\mu$. Then, 
%The convergence of the interest rate to the average rate over time is known as the {\em mean reverting} property of the CIR solution.
Taking expectation in \eqref{CIR}  gives
\begin{equation}
 \label{media}
\mathbb{E}[X(t)]=e^{-\alpha t}\left(\mathbb{E}[X_0]-\mu\right)+\mu,
\end{equation} 
and, since $\alpha>0$,
\begin{equation*}
\lim_{t\to\infty}\mathbb{E}[X(t)]=\mu,   
\end{equation*}
i.e., the parameter $\mu$ is the long-term mean of $X(t)$ and the parameter $\alpha$ represents the speed of convergence. 
In addition, for the CIR process $X(t)$  it can be seen \cite{shreve} that
$$
%\begin{aligned}
\mathbb{E}[X(t)^2]=
\mu^2+\frac{\sigma^2\mu}{2\alpha}+e^{-2\alpha t}\left(\mathbb{E}[X_0^{2}]+\left(2\mu+\frac{\sigma^2}{\alpha}\right)\left(\frac{\mu}{2}-\mathbb{E}[X_0]\right)\right)
+e^{-\alpha t}\left(2\mu+\frac{\sigma^2}{\alpha}\right)\left(\mathbb{E}[X_0]-\mu\right).
$$
Consequently, \begin{equation}\label{media2}
    \lim_{t\to\infty}\mathbb{E}[X(t)^2]=\mu ^2+\frac{ \sigma ^2 \mu }{2 \alpha },
\end{equation}
and the long-term variance of the CIR process is 
$$\lim_{t\to\infty}\mbox{Var}(X(t))=\lim_{t\to\infty}\mathbb{E}[X(t)^2]-\lim_{t\to\infty}\mathbb{E}[X(t)]^2=\frac{ \sigma ^2 \mu }{2 \alpha }.$$

Although the transition density and the distribution of the CIR process are known, a closed form of the solution is not available except when
the parameters of equation \eqref{CIR} satisfy the relation $\sigma^2=4\alpha\mu$,
see \cite{hefter}. 
For this reason, numerical schemes that approximate the solution are required. A number of methods to solve numerically stochastic differential equations have been proposed, see \cite{kloeden} and the references therein. 
Nevertheless, general methods may not work in this problem if they use evaluations of the diffusion coefficient  $\sigma\sqrt{x}$ or of its derivatives, which are not well-defined when negative values appear.
To overcome this difficulty, numerical methods specially designed to solve the CIR equation have been proposed in the literature, see \cite{alfonsi,bossy,dereich,hefter,hmao,lord}.
 A desirable property of any numerical method for solving an SDE is the preservation of qualitative properties of the exact solution. In this sense, our goal is to propose schemes that applied to the CIR problem give numerical approximations $\{X_n\}$ of $X(t)$ that preserve:
 \begin{enumerate}
 \item[(P0)] the non-negativity of the solution, i.e. $X_n\ge 0$,
 \item[(P1)] the mean reverting property, i.e.
  \begin{equation}
    \label{meanprop}
    \lim_{n\to\infty}\mathbb{E}[X_n]=\lim_{t\to\infty}\mathbb{E}[X(t)]=\mu,
\end{equation}
 \item[(P2)] the long-term second moment, i.e.
 \begin{equation}
 \label{m2}
\lim_{n\to\infty}\mathbb{E}[X_n^2]=\lim_{t\to\infty}\mathbb{E}[X(t)^2]=\mu ^2+\frac{ \sigma ^2 \mu }{2 \alpha }.
\end{equation}
 \end{enumerate}
 
%To the best of our knowledge, these three properties of numerical integrators have jointly been studied only for the modified Euler method 
%\begin{equation}\label{HM}
 %   X_{n+1}= X_{n}+\alpha(\mu-X_{n})\Delta+\sigma\sqrt{|X_{n}|}\Delta W_n
 %   \end{equation}
Fulfillment of properties (P1)-(P2) have been studied for the modified Euler method in \cite{hmao} and for a specially designed class of methods in \cite{nuestro}. Higham and Mao, see \cite{hmao}, determined that, under a restriction on the step size, identity
\eqref{meanprop} is verified for the 
modified Euler method.
%scheme \eqref{HM} if $\Delta<2/\alpha$. 
%In addition, given the same restriction on $\Delta$, they show a band where the $\liminf_{n\to\infty}\mathbb{E}[X_n^2]$ and $\limsup_{n\to\infty}\mathbb{E}[X_n^2]$ can be found. 
In \cite{nuestro}, the authors propose a family of methods 
for the numerical solution of the CIR model reproducing the mean-reversion property, as well as a method that captures exactly the first and long-term second moments.

%To improve the modified Euler scheme, which is a low-order convergence method in the weak and strong senses, 
Different numerical schemes have been proposed with better convergence rates than Euler-Maruyama scheme. 
In this work, we consider the family of semi-implicit Milstein, also named $\theta$-Milstein, methods, see \cite{higham}, that for computing numerical solutions of \eqref{sde} take the form
%Given $X_0$, for $n=0,1,2,\dots$
\begin{equation}\label{mil}
    %\begin{aligned}
    X_{n+1}=
    %&\  
    X_{n}+\left\{(1-\theta)f(X_{n})+\theta f(X_{n+1})\right\}\Delta+g(X_n)\Delta W_n
    %\\&
    +\frac{1}{2}g(X_n)g'(X_n)\left(\Delta W_n^2-\Delta\right)
    %\end{aligned}
    \end{equation}
where $\Delta>0$ represents the step-size, and $\theta\in\mathbb{R}$ controls the degree of implicitness. The values $\theta=0$ and $\theta=1$ give the explicit and the fully implicit Milstein methods respectively \cite{mil}. Although it is usual to apply the semi-implicit Milstein methods with values $0\le \theta\le 1$, Higham \cite{higham} showed how the values $\theta>1$ present better stability behavior. 
%\brown{In the following sections, we analyze the convergence of  $\theta$-Milstein schemes with $\theta\ge 1$ applied to solve the CIR model, we show how these schemes may offer benefits in the preservation of properties (P0), (P1), and (P2) respectively, and we present numerical experiments that confirm our theoretical results. }
 In this work, we analyze the suitability of $\theta$-Milstein schemes to solve the CIR model. The remainder of the paper is organized as follows. In Section 2 we study the conditions under which $\theta$-Milstein schemes, when applied to the CIR problem, yield non-negative solutions. Once established the applicability of these schemes, their convergence, both in the strong and weak senses, is addressed in Section 3. Sections 4 and 5 focus on the fulfillment of properties (P1) and (P2) respectively by implicit Milstein methods. Finally, we present in Section 6 numerical experiments that corroborrate our theoretical results. 

\section{Preservation of non-negativity by 
\texorpdfstring{$\theta$}{theta}-Milstein methods}
Given $X_n\ge 0$, the recurrence \eqref{mil} defined by the $\theta$-Milstein method to solve numerically the CIR equation \eqref{CIR} becomes
\begin{equation*}\label{nuestro}
    X_{n+1}=\left(1-\alpha\Delta+\alpha\theta\Delta\right) X_{n}-\alpha\theta\Delta X_{n+1}+\alpha\mu\Delta+\sigma\sqrt{X_n}\Delta W_n+\frac{\sigma^2}{4}\left(\Delta W_n^2-\Delta\right).
    \end{equation*}
Then
\begin{equation}\label{mil2}
    \begin{aligned}
    X_{n+1}=\frac{1}{1+\alpha\theta\Delta}
    %&
    \left\{\left(1-\alpha\Delta+\alpha\theta\Delta\right) X_{n}+\left(\alpha\mu-\frac{\sigma^2}{4}\right)\Delta
   % \right.\\&     \left.\ \ 
   +\sigma\sqrt{X_n}\Delta W_n+\frac{\sigma^2}{4} \Delta W_n^2\right\},
    \end{aligned}
    \end{equation}
which can be written
\begin{equation}
\label{mil3}X_{n+1}=\frac{1}{1+\alpha\theta\Delta}
    \left\{\alpha\Delta(\theta-1) X_{n}+\left(\alpha\mu-\frac{\sigma^2}{4}\right)\Delta+h(W_n)\right\},    
    \end{equation}
where  $h(z)=X_n+\sigma\sqrt{X_n}z+\frac{\sigma^2}{4}z^2$ for $z\in\mathbb{R}$.
The quadratic function $h$ (whose graph is an open upward parabola) attains its absolute minimum at $z^*=-2\sqrt{X_n}/\sigma$. Then for all $z\in\mathbb{R}$
$$h(z)\ge h(z^*)=X_n+\sigma\sqrt{X_n}\,\frac{-2\sqrt{X_n}}{\sigma}+\frac{\sigma^2}{4}\left(\frac{-2\sqrt{X_n}}{\sigma}\right)^2=0.$$
Using this result in \eqref{mil3} gives
$$
X_{n+1}\ge \frac{1}{1+\alpha\theta\Delta}
        \left\{\alpha\Delta(\theta-1) X_{n}+\left(\alpha\mu-\frac{\sigma^2}{4}\right)\Delta\right\}.    
$$
From here:
\begin{proposition}
\label{non-neg}
The $\theta$-Misltein scheme \eqref{mil2} with $\theta\ge 1$ starting at $X_0\ge0$ preserves the non-negativity of the exact solution $X(t)$  if the parameters of the CIR model fulfill the condition
$4\alpha\mu\ge\sigma^2$. In particular, under the Feller condition
%\eqref{feller},
it preserves the positivity of $X(t)$.
\end{proposition}

%\begin{remark}
%Notice also that the numerical solution remains positive if it starts positive when $\theta>1$ and $4\alpha\mu\ge 0$, or $\theta\ge 1$ and $4\alpha\mu> 0$.
%\end{remark}

\begin{remark}
The result for $\theta=1$ was shown in \cite{kahl}. On the other hand, notice the difference between the above results and those in \cite{scal}, where some inaccuracies have been detected. 
\end{remark}

Non-negativity of $X_n$ is a necessary condition for the correct definition of $\theta$-Milstein schemes in \eqref{mil2}. From now on, due to Proposition \eqref{non-neg}, to ensure a meaningful study, we shall focus on $\theta$-Milstein schemes with $\theta\ge 1$ applied to solve numerically CIR problems with parameters fulfilling the condition $\sigma^2\le 4\alpha\mu$.

\section{Strong and weak convergence of \texorpdfstring{$\theta$}{theta}-Milstein methods}\label{secorder}

 It is known that under appropriate conditions on the coefficients $f$ and $g$ of \eqref{sde}, implicit Milstein methods \eqref{mil} are convergent with order 1 in both the strong and weak senses. Unfortunately, this is not the case for equation \eqref{CIR} due to its diffusion coefficient, and a specific study is needed.  Here, using the approach presented in Alfonsi \cite{alfonsi}, we show that the weak order 1 of \eqref{mil} remains, whereas it converges in the strong sense with logarithmic order, that is, \begin{equation}\label{so}    \mathbb{E}\left[\lvert X(t_n)-X^\Delta_{n+1}\rvert\right]=\mathcal{O}\left(\frac{1}{\log{\Delta}}\right).
\end{equation} 
%{\color{blue} quizás quitar. Se han hecho afirmaciones donde esto se da por conocido. Additionally, we prove that \eqref{mil} attains weak order of convergence 1 for a regular enough class of functions. Specifically, if $g:\mathbb{R}_+\rightarrow\mathbb{R}$ has a continuous fourth order derivative with polynomial growth, i.e., for some $A, m>0$ one has $\left|g^{(4)}(x)\right| \leq A\left(1+x^{m}\right)$ for all $x \geq 0$ then \begin{equation}\label{wo}    {\mathbb{E}}\left[g(X(t_n))\right]={\mathbb{E}}\left[g\left({X^\Delta_{n+1}}\right)\right]+\mathcal{O}\left(\Delta\right).\end{equation}}
In the following. we use the extension of big $\mathcal{O}$ notation to  probability theory: for a scheme, $\left(Z_n^\Delta\right)$ we will write $Z_n^\Delta=\mathcal{O}(\Delta^s)$ when $|Z_n^\Delta|/\Delta^s$ has uniformly bounded moments for all sufficiently small $\Delta>0$.

\begin{proposition}\label{boundedtl}
Starting at $X_0{\ge}0$, the implicit $\theta$-Misltein scheme \eqref{mil} with $\theta\ge 1$ applied to solve \eqref{CIR} with  
$4\alpha\mu\ge\sigma^2$ has uniformly bounded moments. 
\end{proposition}
\begin{proof}
We show that ${X}_{{n}}=\mathcal{O}(1)$. It is clear that $\left(X_n\right)$ is adapted and, from Proposition \ref{non-neg}, it is non-negative. Its expression in \eqref{mil2} leads to
\begin{align}
\label{be1}
    X_{n+1}&=\frac{1}{1+\alpha\theta\Delta}
    %&
    \left\{\left(1-\alpha\Delta+\alpha\theta\Delta\right) X_{n}+\left(\alpha\mu-\frac{\sigma^2}{4}\right)\Delta
   +\sigma\sqrt{X_n}\Delta W_n+\frac{\sigma^2}{4} \Delta W_n^2\right\}\\&<\left(1-\alpha\Delta+\alpha\theta\Delta\right) X_{n}+\left(\alpha\mu-\frac{\sigma^2}{4}\right)\Delta  +\sigma\sqrt{X_n}\Delta W_n+\frac{\sigma^2}{4} \Delta W_n^2
    \\&
\le(1+\tau\Delta)X_n+\sigma\sqrt{X_n}\Delta W_n+\mathcal{O}(\Delta)\nonumber
\end{align}
where $\tau=\alpha\,\theta>0$. Since the hypotheses of Lemma 2.6 in \cite{alfonsi} are fulfilled,  we conclude.
\end{proof}

{We now address the weak convergence of the proposed scheme. The proof is strongly based on a  result presented in \cite{alfonsi} that we state here as a lemma.
\begin{lemma}
\label{weakprop}
Let us suppose that $\left(X_{n}^\Delta\right)$ 
 is a nonnegative adapted scheme
such that:
\begin{align}
&X_{n+1}^\Delta={X}_{n}^{\Delta}+\alpha\left(\mu- {X}_{n}^{\Delta}\right)\Delta+\sigma \sqrt{{X}_{n}^{\Delta}}\Delta W_n+m_{t_{n+1}}^{n}-m_{t_{n}}^{n}+\mathcal{O}\left(\Delta^{2}\right)\label{hw1} \\
&\mathbb{E}\left[\left(X_{n+1}^\Delta-{X}_{n}^{\Delta}\right)^{2} \left| \mathcal{F}_{t_{n}}^{\phantom{l}}\right.\right]=\sigma^{2} {X}_{n}^{\Delta} \Delta+\mathcal{O}\left(\Delta^{2}\right)\label{hw2}
\end{align}
where $\{\mathcal{F}_t\}_{t\ge 0}$ denotes the filtration generated by $\{W_t\}_{t\ge 0}$ and the increment $m_{t_{n+1}}^{\Delta}-m_{t_{n}}^{\Delta}$ is a $\mathcal{O}(\Delta)$ martingale.
%$$
%m_{t_{n+1}}^{\Delta}-m_{t_{n}}^{\Delta}=\mathcal{O}(\Delta)
%$$
Then the scheme $(X_n^\Delta)$% verifies \eqref{wo} and \eqref{so}, i.e., it 
converges with weak order 1 and logarithmic strong order.
%$$
%\mathbb{E}\left[f\left({X}_{n}^{\Delta}\right)\right]=\mathbb{E}\left[f\left(X_{t_n}\right)\right]+\mathcal{O}(\Delta)
%$$
\end{lemma}

\begin{proposition}\label{proporder} {Starting at $X_0{\ge}0$, the implicit $\theta$-Misltein scheme \eqref{mil} with $\theta\ge 1$ applied to solve \eqref{CIR} with  
$4\alpha\mu\ge\sigma^2$ attains weak order 1 and logarithmic strong order.}
%for the class of functions described in Propositon \ref{weakprop}.
\end{proposition}
\begin{proof}
{Since we are under the conditions of Proposition \ref{boundedtl}, $X_n=\mathcal{O}(1)$.
First, we show that $X_n$ fulfills \eqref{hw1}. Expanding \eqref{mil2} yields}
\begin{equation}
    \label{hw11}
X_{n+1}=X_n+\alpha(X_n-\mu)\Delta +\sigma\sqrt{X_n} \Delta W_n +m_{n+1}-m_n+\mathcal{O}(\Delta^2)\end{equation}
where $m_n$ is the discrete $\mathcal{F}_{t_{n-1}}$-adapted process  
\begin{equation*}
m_n=
\begin{cases}
 0 & n=0 \\
 m_{n-1}-\sigma  \left( \alpha  \theta  \sqrt{X_{n-1}}\Delta  \Delta W_n +\frac{\sigma}{4}  \left(\Delta -\Delta W_n^2\right)\right) & n>0\\
\end{cases}
\end{equation*}
Since $m_{n+1}-{m}_{n}=\mathcal{O}(\Delta)$ and $\mathbb{E}\left[m_{n+1}-{m}_{n} \mid \mathcal{F}_{t_{n}}\right]=0$, \eqref{hw1} holds.

To calculate the conditional expectation of $\left(X_{n+1}^\Delta-{X}_{n}^{\Delta}\right)^{2}$ up to $\mathcal{O}(\Delta^2)$ it is sufficient to square the  expansion \eqref{hw11}, remove the terms of null expectation (the terms that are a multiple of an odd power of $\Delta W_n$) and group the terms of $\mathcal{O}(\Delta^2)$ together. 
So \eqref{hw2} is obtained, and the result follows from Lemma \ref{weakprop}.
\end{proof}
}
    
\section{\label{secm1} Mean-reverting \texorpdfstring{$\theta$}{theta}-Milstein schemes}

We analyze now under which  conditions the schemes  \eqref{mil2} with $\theta\ge 1$ are able to reproduce  the mean-reverting property of the exact solution
for problems with the condition $\sigma^2\le 4\alpha\mu$.

Since $\mathbb{E}\left[\Delta{W_{n}}\right]=0$ and $\mathbb{E}\left[\Delta{W_{n}}^2\right]=\Delta$, taking expectations in \eqref{mil2} gives 
\begin{equation}
\label{e1}
\mathbb{E}\left[X_{n+1}\right]=A\,\mathbb{E}\left[X_{n}\right]+B
\end{equation}
where
\begin{equation}
    \label{defab}
    A:=\frac{1-\alpha\Delta(1-\theta)}{1+\alpha\theta\Delta},\ \ 
    B:=\frac{\alpha\mu\Delta}{1+\alpha\theta\Delta}.
\end{equation}
From \eqref{e1}
\begin{equation}\label{exn}
    \mathbb{E}\left[X_{n}\right]=A^n \mathbb{E}\left[X_0\right]+
    \frac{1-A^n}{1-A}\, B=
    A^n\left(\mathbb{E}\left[X_0\right]-\mu\right)+\mu,
\end{equation}
where we have used that $B/(1-A)=\mu$. Since $\theta\ge 1$,
$$0<\frac{1+\alpha\Delta(\theta-1)}{1+\alpha\theta\Delta}=A=1-\frac{\alpha\Delta}{1+\alpha\theta\Delta}<1.$$
Then $A^n\to 0$ as $n\to\infty$, and from \eqref{exn}, 
%when $|A|<1$ 
\begin{equation*}    \lim_{n\to\infty}\mathbb{E}\left[X_{n}\right]=
%       \frac{B}{1-A}=
\mu.
\end{equation*}
We have proved:
\begin{theorem}
\label{teo}
Given any $\Delta>0$, the $\theta$-Milstein scheme with $\theta\ge 1$ preserves the long-term mean of the exact solution of any CIR equation with $4\alpha\mu-\sigma^2\ge 0$.
\end{theorem}

Notice the similarity between the expressions of the exact mean \eqref{media} and the numerical mean \eqref{exn} calculated with the $\theta$-Milstein method.
The difference between them at $t_n=n\Delta$ can be written
$$\varepsilon_n^\Delta(\theta):=\mathbb{E}\left[X_{n}\right]-\mathbb{E}\left[X(t_n)\right]=(A^n-e^{-\alpha \Delta n})\left( \mathbb{E}\left[X_0\right]-\mu\right)=g(\theta)\left( \mathbb{E}\left[X_0\right]-\mu\right)$$
where
$$g(\theta):=\left(1-\frac{\alpha\Delta}{1+\alpha\theta\Delta}\right)^n-e^{-\alpha \Delta n}.$$

Since $\alpha\Delta>0$, we have 
$e^{\alpha\Delta}>1+\alpha\Delta$; therefore
$ e^{-\alpha\Delta}<(1+\alpha\Delta)^{-1}$ and
from here $g(1)=(1+\alpha\Delta)^{-n}-e^{-\alpha \Delta n}>0.$
%We know that $g(\theta)\to 0$ as $n\to\infty$, but we may ask if there exist $\theta\ge 1$ such that the error $\left|\varepsilon_n(\theta)\right|$ is minimum for any fixed $n\in\mathbb{N}$. Clearly, this task is equivalent to minimizing $|g(\theta)|$ in $[0,+\infty)$. Since
As $$g'(\theta)=\frac{\alpha ^2 \Delta ^2 n \left(1-\frac{\alpha  \Delta }{\alpha  \Delta  \theta +1}\right)^{n-1}}{(\alpha  \Delta  \theta +1)^2}>0,$$
$g$ is an increasing function. 
%On the other hand, since $\alpha\Delta>0$ we have 
%$e^{\alpha\Delta}>1+\alpha\Delta$, and this implies that
%$ 1/(1+\alpha\Delta)>e^{-\alpha\Delta}$; then
%$$g(1)=\frac{1}{(1+\alpha\Delta)^n}-e^{\alpha\Delta n}>0.$$
%To sum up, $g(\theta)$ is an increasing function, which is positive for $\theta=1$. 
Consequently, for any $\theta\ge 1$, $g(\theta)\ge g(1)>0$.
Then, on the interval $[1,\infty)$ we have that $|g(\theta)|=g(\theta)$ and $g$ attains its minimum value  at $\theta=1$. This proves that
the fully implicit Milstein method  gives the best approximation of the mean value of the exact solution: 
%and $|g(\theta)|=g(\theta)$ attains its minimum value at $\theta=1$. This proves that:
\begin{proposition}
    \label{minerror1}
    For $\theta\geq 1$, the error at each point $t_n$    $$\left|\varepsilon_n^\Delta(\theta)\right|=\left|\mathbb{E}\left[X(t_n)\right]-\mathbb{E}\left[X_{n}\right]\right|$$ attains its minimum when $\theta=1$.
\end{proposition}

\section{Long-term second moment}\label{secm2}

In this section, we explore if the $\theta$-Milstein methods with $\theta\ge 1$ also retain the long-term second moment of a CIR problem with parameters fulfilling the condition $\sigma^2\le 4\alpha\mu$, i.e. if \eqref{m2} holds 
when $X_n$ is obtained recursively by \eqref{mil2}.
Squaring \eqref{mil2} and taking expected values, we get 
\begin{equation}
\begin{aligned}
    \mathbb{E}[X_{n+1}^2]=&\frac{1}{(1+\alpha\theta\Delta)^2}
    \left\{\left(1-\alpha\Delta+\alpha\theta\Delta\right)^2 \mathbb{E}[X_{n}^2]+\left(\alpha\mu-\textstyle\frac{\sigma^2}{4}\right)^2\Delta^2\right.
    \\&\ \ +\sigma^2\mathbb{E}[X_{n}]\Delta +\frac{\sigma^4}{16} 3 \Delta^2+
    2(1-\alpha\Delta+\alpha\theta\Delta)\left(\alpha\mu-\textstyle\frac{\sigma^2}{4}\right)\Delta\, \mathbb{E}[X_{n}]
    \\&    \left.\ \ 
    +2(1-\alpha\Delta+\alpha\theta\Delta)\textstyle\frac{\sigma^2}{4}\Delta\, \mathbb{E}[X_{n}]+
    2\left(\alpha\mu-\textstyle\frac{\sigma^2}{4}\right)
       \textstyle\frac{\sigma^2}{4}\Delta^2
    \right\},
    \end{aligned}
    \end{equation}
which can be written
%sing that $\mathbb{E}\left[\Delta{W_{n}}^4\right]=3\Delta^2$ we have that 
\begin{equation}
    \label{exn+1}
    \mathbb{E}[X_{n+1}^2]=A^2\, \mathbb{E}[X_n^2]+D\, \mathbb{E}[X_n]+ E
\end{equation}
where $A$ is given in \eqref{defab} and
$$D:=\frac{\left(\sigma^2+2\alpha\mu(1-\alpha\Delta+\alpha\theta\Delta)\right)\Delta}{(1+\alpha\theta\Delta)^2},\ \ 
E:=\frac{\left(8\alpha^2\mu^2+\sigma^4\right)\Delta^2}{8(1+\alpha\theta\Delta)^2}.$$
Recall that $\theta\ge 1$ implies $0<A<1$ and  $\lim\limits_{n\to\infty}\mathbb{E}\left[X_{n}\right]=
\mu$. Then, from \eqref{exn+1},
\begin{equation}
   \lim_{n\to\infty} \mathbb{E}\left[X_{n}^2\right]=\frac{D\mu+E}{1-A^2}= \mu ^2+\frac{ \mu\sigma ^2 }{2 \alpha }+ \frac{ \sigma^2(4\mu\alpha(1-2\theta)+\sigma^2)\Delta}{8 \alpha (2+\alpha\Delta(2\theta-1)) }.
  \end{equation}
This equality shows that the
approximation $\{X_{n}\}$ holds \eqref{m2}
if and only if
$4\mu\alpha(1-2\theta)+\sigma^2=0$, which is equivalent to
$$\theta=\frac{\sigma^2+4\mu\alpha}{8\mu\alpha}.$$
Using that $\sigma^2\le 4\alpha\mu$ and $\theta\ge 1$ we conclude that:
\begin{theorem}
\label{teo2}
The fully implicit Milstein scheme ($\theta=1$) preserves the long-term second moment of the exact solution of the CIR equation if and only if the parameters fulfill $4\alpha\mu=\sigma^2$. 
For CIR problems with $4\alpha\mu>\sigma^2$, $\theta$-Milstein methods are not able to preserve the exact long-term second moment.
\end{theorem}

For $\theta>1$ and $\sigma^2\le 4\alpha\mu$ or $\theta=1$ and $\sigma^2< 4\alpha\mu$, one has $4\mu\alpha(1-2\theta)+\sigma^2<0$; then the error at $t_n$ of the $\theta$-Milstein method applied with step-size $\Delta>0$  in the calculation of the second moment is
\begin{equation*}\label{Second Moment Error}
   \varepsilon_{2,n}^\Delta(\theta):=
\mathbb{E}\left[X_{n}^2\right]-
 \mathbb{E}\left[X^2({t_n})\right],
\end{equation*}
and when time tends to infinity we have
$$\varepsilon_{2}^\Delta(\theta):=\lim_{n\to\infty}\varepsilon_{2,n}^\Delta(\theta)=
\frac{ \sigma^2(4\mu\alpha(1-2\theta)+\sigma^2)\Delta}{8 \alpha (2+\alpha\Delta(2\theta-1))}<0,$$
i.e., when $\theta>1$ or $\sigma^2<4\alpha\mu$ then $\theta$-methods underestimate the exact long-term second moment with negative bias $\lim_{n\to\infty}\varepsilon_{2,n}^\Delta(\theta)$ that tends to zero as the step size $\Delta$ does so. We can analyze for which values of $\theta$ the long-term second moment error, $\varepsilon_{2}^\Delta(\theta)$, attains its minimum (as a function of $\theta$). We have that $\varepsilon_{2}^\Delta(\theta)<0$ and $$(\varepsilon_{2}^\Delta)'(\theta)=-\frac{\Delta  \sigma ^2 \left(\Delta  \sigma ^2+8 \mu \right)}{4 (2+\alpha  \Delta  (2 \theta -1))^2}<0$$  from which the next proposition follows:
\begin{proposition}
    \label{minerror2}
    For $\theta\geq 1$, the error $$\left|\varepsilon_2^\Delta(\theta)\right|=\lim_{n\to\infty}\left|\mathbb{E}\left[X(t_n)\right]-\mathbb{E}\left[X_{n}\right]\right|$$ attains its minimum when $\theta=1$. 
\end{proposition}

\section{Numerical Experiments}\label{secnum}

 Here, the theoretical results presented in Sections \ref{secorder}-\ref{secm2} are confirmed with numerical experiments. 
For the sake of comparison, in addition to $\theta$-Milstein schemes with $\theta=1$ and $\theta=1.5$, we shall use a series of schemes specially designed for the integration of the CIR problem: the modified Euler (HM) scheme presented in \cite{hmao}, the drift-implicit (DI) and E(0) schemes proposed by Alfonsi in \cite{alfonsi}, the Milstein-like scheme (HH) given by Hefter and Herzwurm in \cite{hefter},  as well as the mean reverting method (MS) proposed by the authors in \cite{nuestro}. 

In the experiments, we have integrated the equation \eqref{CIR} with different sets of parameters and initial value:
\begin{align}
        \alpha &=0.43,  \quad\mu =0.06, \quad\sigma =0.15,   \quad X_0 =0.057,\label{par1} \\
        \alpha &=0.5,  \ \, \quad\mu =0.5, \ \, \quad\sigma =1, \ \ \ \,\quad X_0 =0.525. \label{par2}
    \end{align}
The set of parameters given in \eqref{par1} was proposed in \cite{KladivkoMLE} as the maximum likelihood estimation values. Notice that they fulfill the condition $\sigma^2<4\alpha\mu$. The parameters in \eqref{par2} are much larger than those  in \eqref{par1} and the starting point is greater than the long-term mean. Notice also that they fulfill $\sigma^2=4\alpha\mu$. 

\medskip
\emph{Experiment 1}

\noindent The first experiment is devoted to confirm the weak order result proved in Proposition \ref{weakprop}.
We estimate the weak error ($\varepsilon_X$) of $\theta$-Milstein method, $\theta=1$,  integrating numerically the equation \eqref{CIR} with parameters \eqref{par1} and step sizes $\Delta=2^{-1},2^{-2},...,2^{-8}$ in the time interval $[0,1]$. For each calculation, we have used $10^6$ paths of the solution.  Analogous calculations were carried out with the methods HH, HM, E(0), DI and MS. In the top plot of Figure \ref{exp1}, the values of $\log(\varepsilon_{X}^\Delta)$ vs. $\log(\Delta)$ obtained with each method are represented. In order to contrast, dashed lines with slopes 1/2 and 1 have also been plotted. 

\begin{figure}
    \centering  \includegraphics[width=0.7\textwidth]{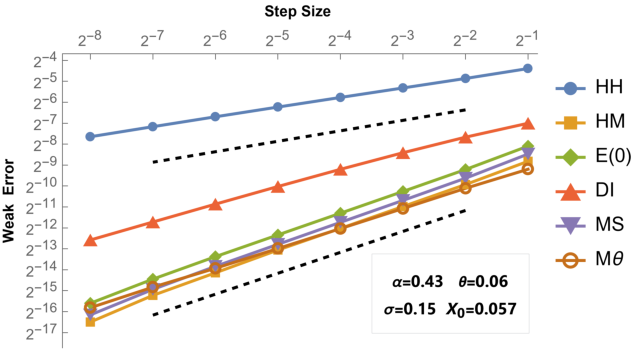}    \includegraphics[width=0.7\textwidth]{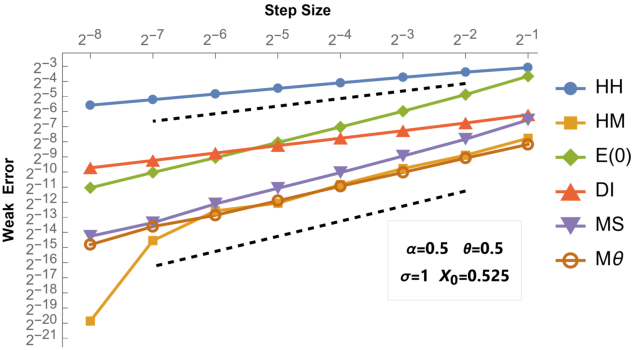}
    \caption{Log-log plot of the weak error $\varepsilon_X^\Delta$ at time $t = 1$ against $\Delta$ for schemes HH,  HM, E(0), DI, MS and implicit Milstein ($\theta=1$) with
data \eqref{par1} (top) and  \eqref{par2} (bottom).}
    \label{exp1}
\end{figure}
We have repeated the experiment for parameters and initial value given in \eqref{par2} and the results are shown in the bottom plot of Figure \ref{exp1}. In both cases, it can be seen that the line corresponding to the implicit Milstein scheme is similar to a straight line with slope approximately equal to 1. In addition, it can be observed that, except for HM with parameters \eqref{par1}, and for HM and HH with parameters \eqref{par2}, the remainder of the methods have graphs with similar slopes.

\medskip

\emph{Experiment 2}

\noindent
Here we illustrate the result regarding the strong convergence of \eqref{mil} presented in Proposition \ref{weakprop}. 
{Since HH has been shown to have strong order ${1}/{2}$ in the integration of the CIR model (with any set of parameters), to compute the strong errors  ($\mathcal{E}_X$) we use its numerical solution calculated with step size $\Delta=2^{-15}$ in place of the exact solution. }
We use $10^5$ paths to calculate ($\mathcal{E}_X^\Delta$) using the same schemes, step sizes and parameters as in Experiment 1. 
 Figure \ref{exp2} shows the values of $\log(\mathcal{E}_{X}^\Delta)$ vs. $\log(\Delta)$ obtained with each method for the data \eqref{par1} (top plot) and \eqref{par2} (bottom plot).
\begin{figure}
    \centering    \includegraphics[width=0.7\textwidth]{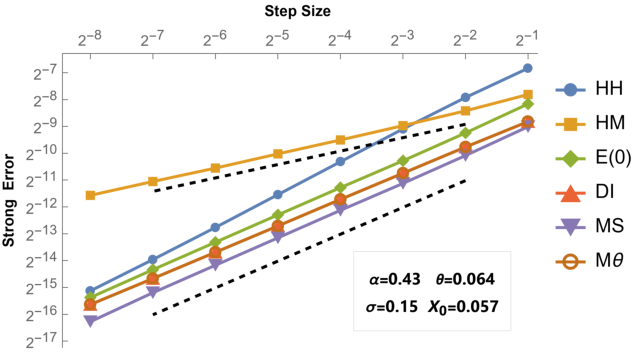}  \includegraphics[width=0.7\textwidth]{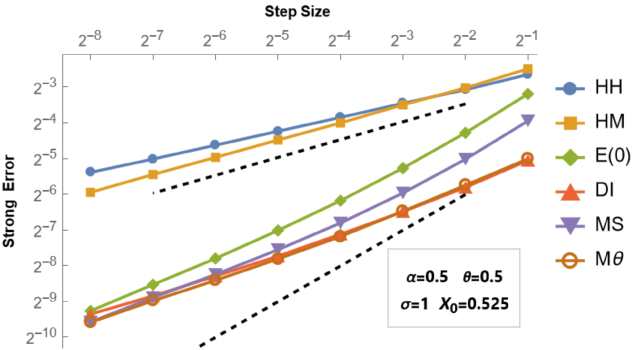}
    \caption{Log-log plot of the strong error $\mathcal{E}_X^\Delta$ at time $t = 1$ against $\Delta$ for schemes HH,  HM, E(0), DI, MS and implicit Milstein ($\theta=1$) with
data \eqref{par1} (top plot) and  \eqref{par2} (bottom plot).}
    \label{exp2}
\end{figure}
For the sake of comparison, dashed lines with slopes 1/2 and 1 have been represented. 
The Milstein method integration produces a line with an average slope of 0.98 for the data \eqref{par1} (top plot) and 0.66 for the data \eqref{par2} (bottom plot), suggesting that the strong order of convergence of Milstein method can be greater than what was shown in Proposition \eqref{proporder}. In Figure \ref{exp2} it can also be seen that all schemes except HM attain order (slope) near to 1 when they are applied to the CIR problem with parameters \eqref{par1} (top) whereas none do it when the parameters \eqref{par2} are considered (bottom).

%The pathwise closeness of an approximation $\{X^\Delta(t)\}$ is  measured by the absolute error at each point $t$ \begin{equation}\label{abserror}    \varepsilon_{X^\Delta}(t):=\mathbb{E}|X_t-X^\Delta(t)|;\end{equation}this mean is calculated repeating $M$ different simulations for the exact {and  numerical solutions corresponding} to thesame sample paths of the Brownian process and computing$$\varepsilon_{X^\Delta}(t)=\frac{1}{M}\sum_{k=1}^{M}|{X}_{t,k}-X^\Delta_{k}(t)|.$$}In these experiments we simulate for each method $N=10^5$ paths of the numerical solution of \eqref{CIR} with step sizes $\Delta=2^{-1},2^{-2},...,2^{-8}$ along the interval $[0,1]$ with parameters and initial value given as\begin{equation   \alpha =0.3,  \quad\mu =0.5, \quad\sigma =0.1,   \quad X_0 =0.6. \label{par01}    \end{equation}

\medskip
\emph{Experiment 3}

\noindent
We study the preservation of the long-term mean by the implicit Milstein schemes. To do so, we simulate for $N=3\times10^6$ paths of the numerical solution given by the $\theta$-Milstein schemes, with $\theta=1,\,1.5$, step size $\Delta=1/8,$ and data \eqref{par1} to equation \eqref{CIR} along the interval $[0,15]$. In order to contrast, we do the same with schemes HM, DI, E(0), and HH.

%In the first experiment the parameters, initial value, and step size are given as\begin{equation}\alpha =0.43,  \quad\mu =0.06, \quad\sigma =0.15,   \quad X_0 =0.057, \quad\Delta =1/8.\label{par1}    \end{equation}

 For each scheme we calculate the sample mean $\overline{X}_{n}$ and the distance from $\overline{X}_{n}$ to the long-term mean as
\begin{equation*}
\overline{X}_{n}:=\frac{1}{N}\textstyle\sum\limits_{k=1}^{N}{X}_{n,k};\quad d_{\overline{X}}(t_{n}):=\left|\overline{X}_n-\mu\right|.
\end{equation*}
In Figure \ref{fig1} the graphical representation of the sample mean (top) and its distance to the long-term mean (bottom) for each method are shown. One can observe that only the modified Euler and $\theta$-Milstein methods revert to the long-term mean, in accordance with the theoretical results of Section \ref{secm1}.
\begin{figure}[htb]
\centering
\includegraphics[width=0.7\textwidth]{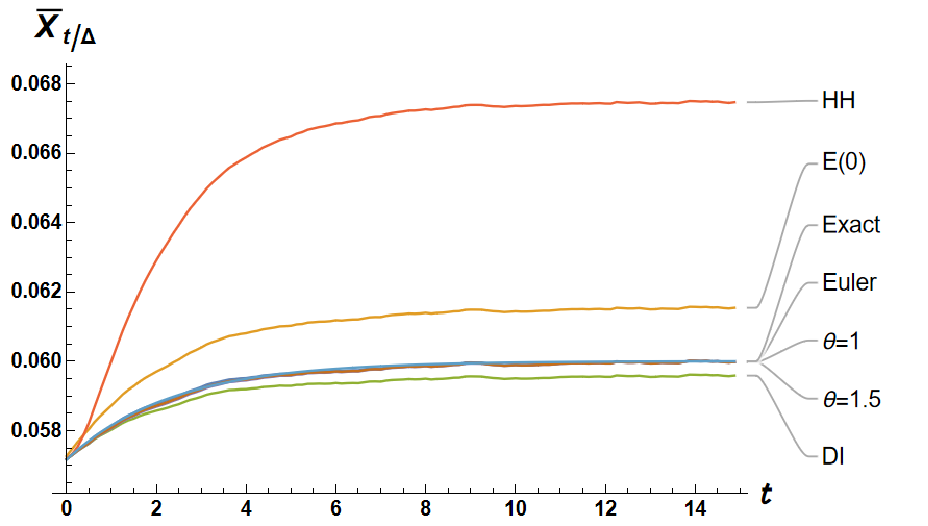}\\
\includegraphics[width=0.7\textwidth]{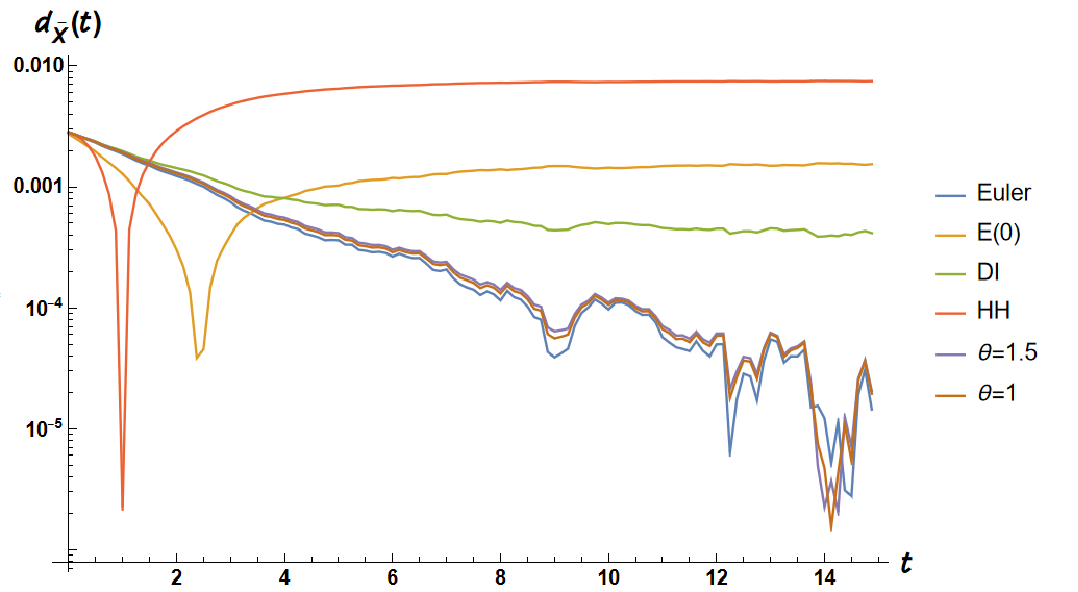}
\caption{Evolution of the first sample moment (top), and its distance to the long-term mean (bottom) for solving \eqref{CIR} with parameters \eqref{par1} and schemes \eqref{mil2}, HM, DI, E(0), and HH.} 
\label{fig1}    
\end{figure}

\medskip
\emph{Experiment 4}

\noindent
We repeat the calculations of the third experiment for the parameters given in \eqref{par2} (plots not shown as they are similar to Figure \ref{fig1}) and also compute for each method the sample second moment $\overline{X^2_{n}}$ and the distance from $\overline{X^2_{n}}$ to the long-term second moment as
\begin{equation*}
\overline{X^2_{n}}:=\frac{1}{N}\sum_{k=1}^{N}{X}^2_{n,k};\quad d_{\overline{X^2}}(t_{n}):=\textstyle\left|\overline{X^2_n}-\left(\mu^2+\frac{\sigma^2\mu}{2\alpha}\right)\right|.
\end{equation*}
In Figure \ref{fig2} the graphical representations of  $\overline{X^2}_{n}$ and $d_{\overline{X^2}}(t_{n})$ are shown on the top and bottom respectively for each method
. As can be seen, the fully implicit Milstein scheme is the only one whose sample second moment converges to the long-term second moment of the CIR process. This is in accordance with Theorem \ref{teo2}, as the parameters verify $\sigma^2=4\alpha\mu$. One can also observe that the sample second moment of the $\theta-$Milstein scheme with $\theta=1.5$ presents the second-nearest distance to the long-term second moment. %The plots for the first sample moment are not included, as they show the same behavior as in the first experiment, Figure \ref{fig1}.
\begin{figure}[htb]
\centering
\includegraphics[width=.7\textwidth]{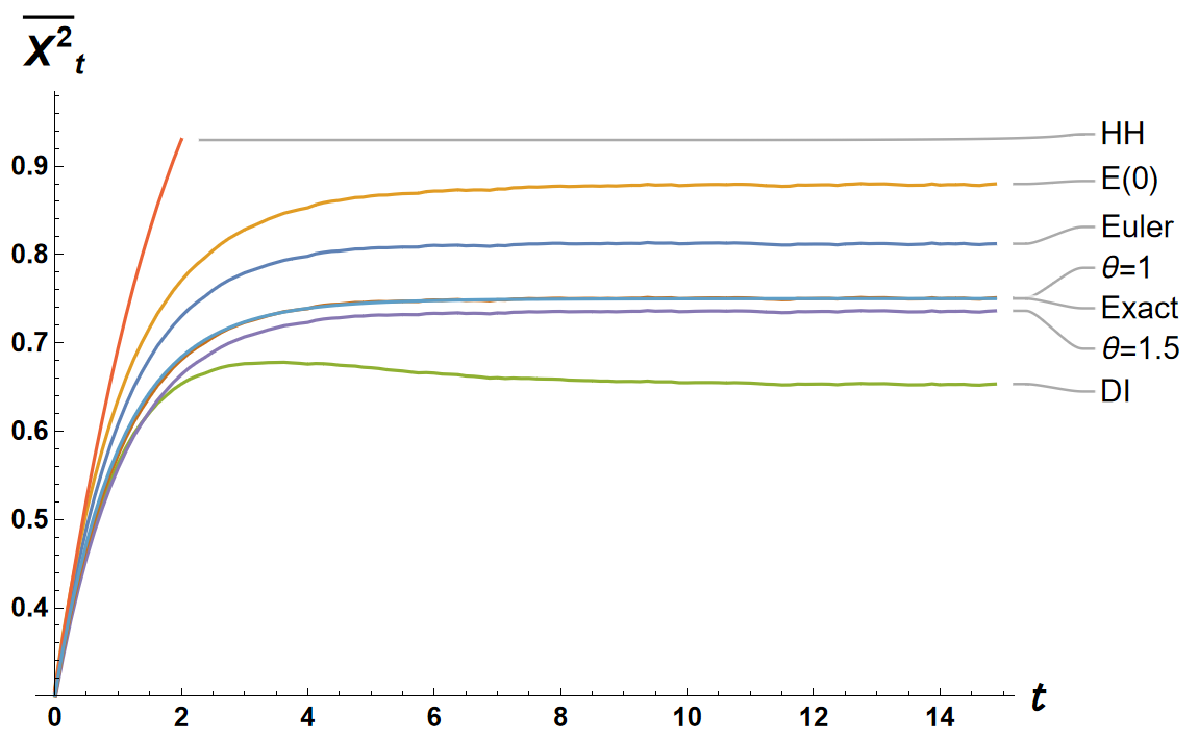}\\
\includegraphics[width=.7\textwidth]{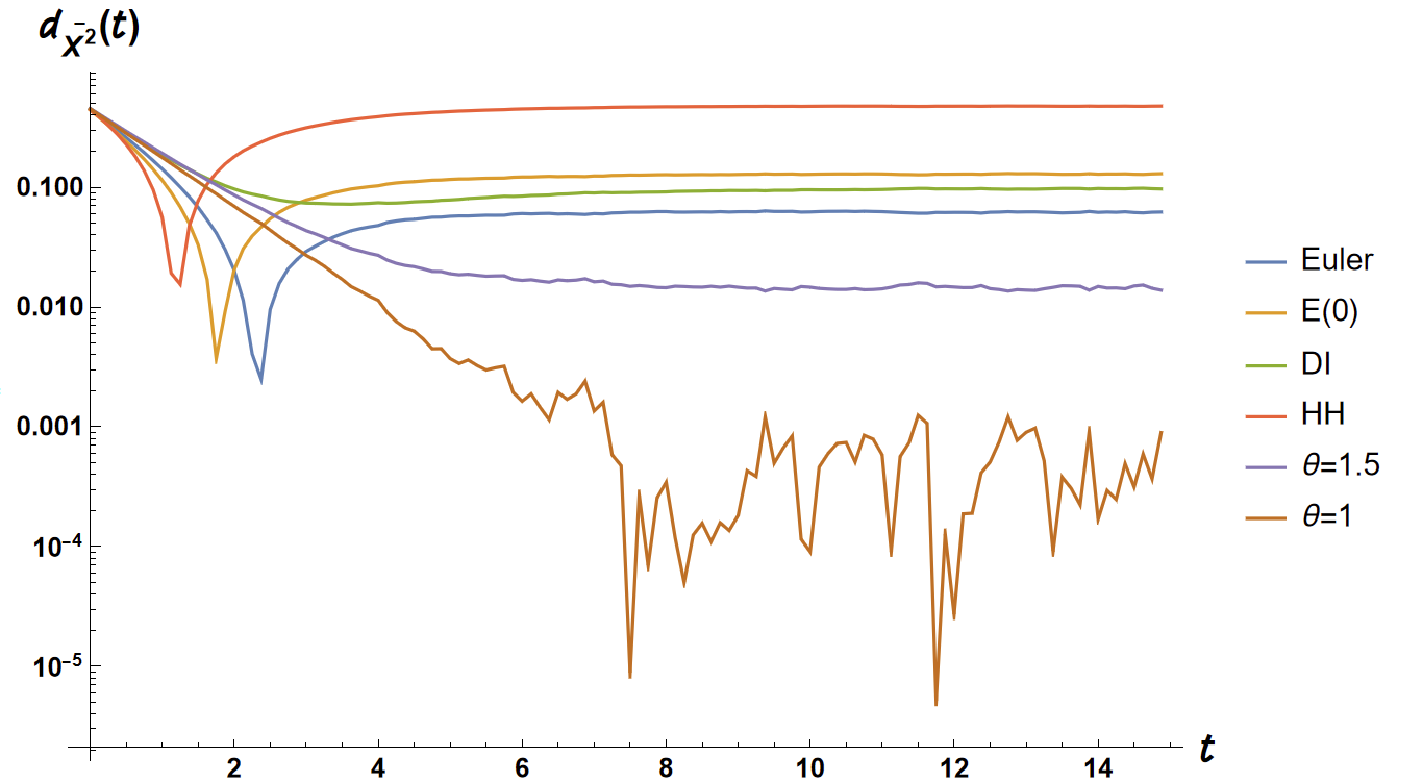}
\caption{Evolution of the sample second moment (top) and its distance to the long-term second moment (bottom) for solving \eqref{CIR} with parameters \eqref{par1} and schemes \eqref{mil2}, HM, DI, E(0), and HH.}
\label{fig2}
\end{figure}

\medskip
\emph{Experiment 5}

\noindent
Finally, we use equation \eqref{exn} to calculate the first moment error $\varepsilon_n^\Delta(\theta)$ of the $\theta-$Milstein schemes,  for the data \eqref{par1}; and equation \eqref{exn+1} to calculate the second moment error $\varepsilon_{2,n}^\Delta(\theta)$ for the data \eqref{par2}. Both errors are calculated for $\theta\in\{1,1.25,1.5,\ldots, 3\}$ and represented graphically in Figure \ref{fig3}:
\begin{figure}
\includegraphics[width=.515\textwidth]{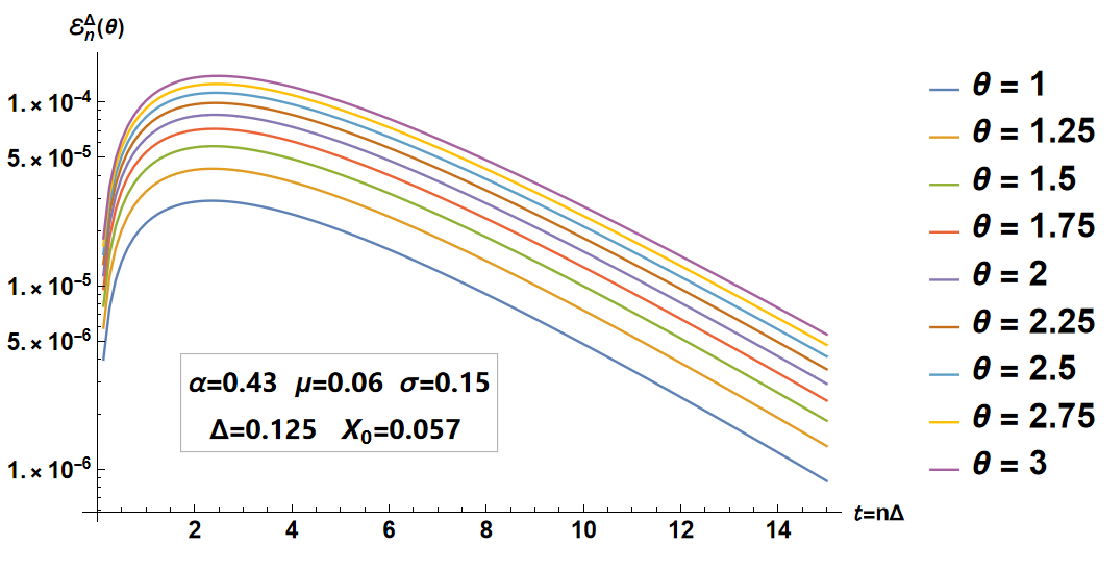}
\includegraphics[width=.48\textwidth]{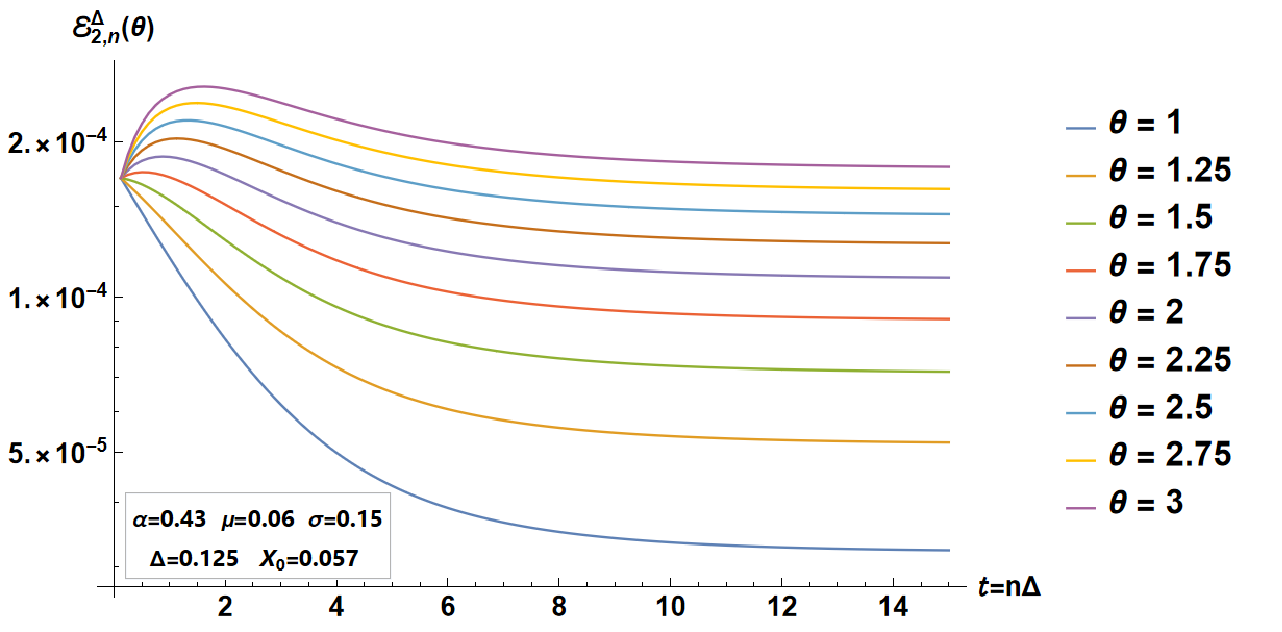}
\includegraphics[width=.515\textwidth]{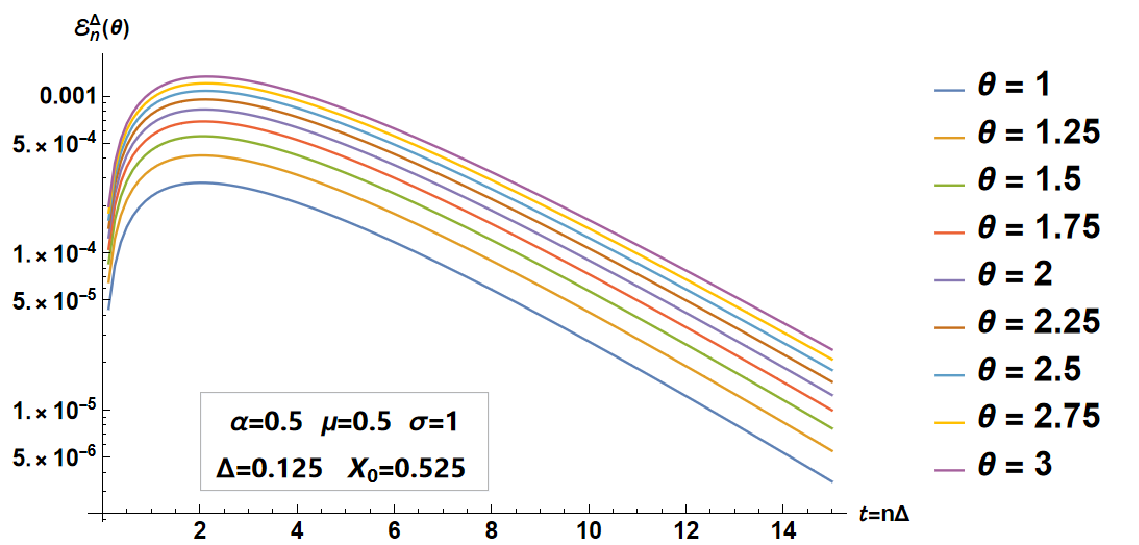}
\includegraphics[width=.48\textwidth]{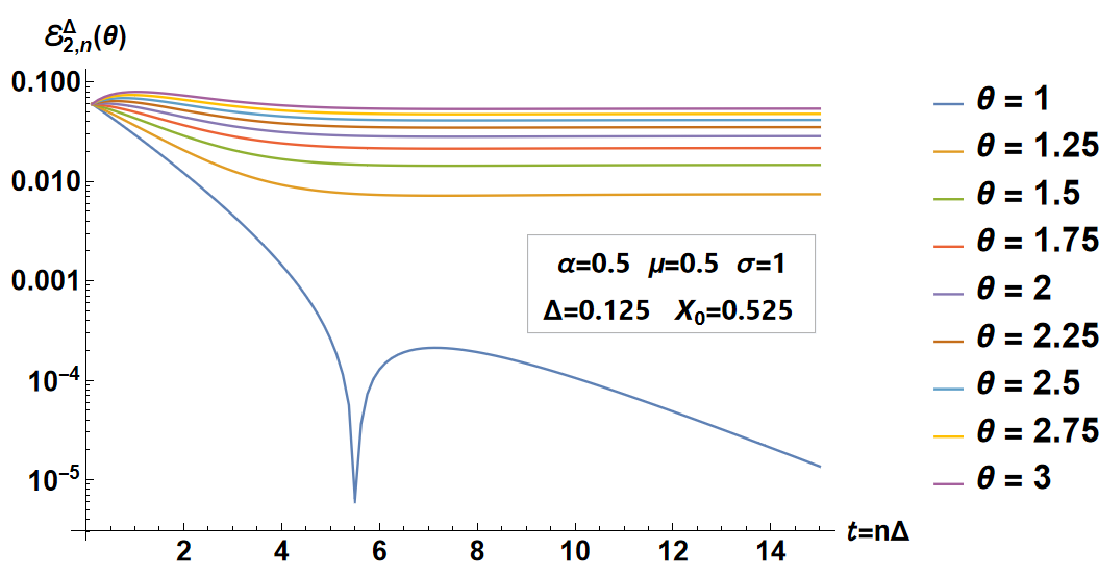}
\caption{First (left) and second (right) moment errors of implicit $\theta-$Milstein schemes {applied to the model with parameters \eqref{par1} (top) and \eqref{par2} (bottom)}.}
\label{fig3}
\end{figure}
The plot on the left shows that the first moment error converges to $0$ for all values of $\theta$ and both sets of parameters. The plot on the right shows, in agreement with Theorem \ref{teo2}, that only the $\theta-$Milstein scheme with $\theta=1$ preserves the long-term second moment for the data verifying $\sigma^2=4\alpha\mu$.

\section{Conclusions}
  The implicit $\theta$-Misltein schemes with $\theta\ge 1$ are appropriate methods to capture qualitative properties of the solution of the CIR model under the parameter condition $4\alpha\mu\ge\sigma^2$. They preserve the non-negativity of the exact solution as well as the mean-reverting property. The method for $\theta=1$ gives the least error with respect to the mean and long-term second moment. Furthermore, it can preserve the exact long-term second moment when $4\alpha\mu=\sigma^2$. Regarding the convergence, we have proved that the studied methods converge in the strong and weak senses when applied to the CIR equation. The weak order of convergence was proved (and confirmed numerically) to be 1. The strong order was proved to be at least logarithmic and numerical experiments suggest that it may attain higher order (up to 1).


\begin{thebibliography}{99}

\bibitem{alfonsi}
Alfonsi A.:  On the discretization schemes for the CIR (and Bessel squared) processes. Monte Carlo Methods Appl. 11, 355-384 (2005).
%DOI: 10.1515/156939605777438569 

\bibitem{bossy}
Bossy M., Olivero, H.: Strong convergence of the symmetrized Milstein scheme for some CEV-like SDEs. Bernoulli J., 24, 1995–2042 (2017).
%\bibitem{chou}
%Choudhry, M; Lizzio, M. Advanced Fixed Income Analysis, 2nd Edition, Elsevier, 2015.

\bibitem{cox}
Cox J., Ingersoll J., Ross S.:  A theory of the
term structure of interest rates. 
Econometrica, 53(2), 385-407 (1985).



%\bibitem{max}
%Deelstra G., Delbaen F, 1998. Convergence of discretized stochastic (interest rate) processes with stochastic drift term. Applied Stochastic Models Data Analysis. 14, 77-84.

\bibitem{dereich}
Dereich S., Neuenkirch A., Szpruch L.:  An Euler-type method for the strong approximation of the Cox–Ingersoll–Ross process. Proc. R. Soc. A. 468, 1105–1115 (2012). 

\bibitem{feller}
Gikhman, I.I.: A short remark on Feller’s square root condition (2011). %http://dx.doi.org/10.2139/ssrn.1756450

\bibitem{hefter}
Hefter, M., Herzwurm, A.: Optimal strong approximation of the one-dimensional squared Bessel process. Commun. Math. Sci. 15(8), 2121–2141 (2017).

\bibitem{hefter2}
Hefter M., Herzwurm A.: Strong convergence rates for Cox–Ingersoll–Ross processes—full parameter range. J. Math. Anal. Appl. 459(2), 1079-1101 (2018). 

\bibitem{higham}
Higham, D.: A-stability and stochastic mean-square stability. {BIT} {40}, 404-409 (2000).

\bibitem{hmao}
Higham D., Mao X.: Convergence of Monte Carlo simulations involving the mean-reverting square root process. J. Comput. Financ. 8(3), 35-61 (2005). 
%DOI: 10.21314/jcf.2005.136

\bibitem{kahl}
Kahl, C., Günther, M., Rossberg, T.:
Structure preserving stochastic integration schemes in interest rate derivative modeling,
Appl. Numer. Math. 58 (3)
284-295 (2008).


\bibitem{KladivkoMLE}
Kladívko, K.: Maximum likelihood estimation of the Cox-Ingersoll-Ross process: the Matlab implementation. {Technical Computing Prague}, 7(8), 1-8 (2007).

\bibitem{kloeden}
Kloeden P., Platen E.: {Stochastic Differential Equations}. Springer,  Heidelberg (1992). 
%DOI: 10.1007/978-3-662-12616-5

\bibitem{lord}
Lord R., Koekkoek R., van Dijk D.: A comparison of biased simulation schemes for stochastic volatility models. Quantitative Finance, 
10, 177-194  (2010).

\bibitem{nuestro}
Llamazares-Elias, S., Tocino, A.:  Mean-reverting schemes for solving the CIR model, Journal of Computational and Applied Mathematics 434 (2023) 115354.

\bibitem{mao}
Mao X.: {Stochastic Differential Equations and Applications}. Elsevier (2007).

\bibitem{mil}
Milstein, G.N.: 
Numerical Integration of Stochastic Differential Equations.
Kluwer Academic Publishers (1995). 

\bibitem{oksendal}
Øksendal B.: {Stochastic differential equations}. Springer (2003). 

\bibitem{scal}
Scalone, C.: 
Positivity preserving stochastic $\theta$-methods for selected SDEs.
Appl. Numer. Math. 172, 351-358 (2022).



\bibitem{shreve}
Shreve, S. E.: Stochastic calculus for finance II: Continuous-time models. Springer (2004).

\bibitem{yamada}
Yamada Y., Watanabe S. On the uniqueness of solutions of stochastic
differential equations. J. Math. Kyoto Univ. 11, 155-167 (1971). 
\end{thebibliography}
\end{document}